\documentclass[11pt]{amsart} 
\usepackage{amsmath}
\usepackage{amsfonts}
\usepackage{amssymb}
\usepackage{amsthm}
\usepackage{amsthm}
\usepackage{graphics}
\usepackage{color}
\usepackage{algorithm}
\usepackage{algorithmic}

\input xy
\xyoption{all}

\newtheorem{theorem}{Theorem}[section]
\newtheorem{lemma}[theorem]{Lemma}

\newtheorem{corollary}[theorem]{Corollary}
\newtheorem{conjecture}[theorem]{Conjecture}

\theoremstyle{definition}

\theoremstyle{remark}

\numberwithin{equation}{theorem}

\newcommand{\qform}[1]{{\left\langle{#1}\right\rangle}}

\newcommand{\inverselimit}[1]{\displaystyle \lim_{\leftarrow}}

 \newcommand{\cF}{{\mathcal F}}

\newcommand{\cL}{{\mathcal L}}

\DeclareMathOperator{\id}{id}

\DeclareMathOperator{\ind}{ind}

\DeclareMathOperator{\End}{End}

\newcommand{\sq}[1]{#1^\times/#1^{\times 2}}

\DeclareMathOperator{\Ad}{Ad}

\DeclareMathOperator{\ad}{ad}

\begin{document}

\title{Involutions, odd-degree extensions and generic splitting}

\author{Jodi Black and Anne Qu\'eguiner-Mathieu}

\address{Department of Mathematics, Bucknell University, 380 Olin Science Building,  Lewisburg, Pennsylvania, 17837}

\email{jodi.black@bucknell.edu}

\address{Universit\'e Paris 13, Sorbonne Paris Cit\'e, LAGA, CNRS (UMR 7539), 99 avenue
 Jean-Baptiste Cl\'ement, F-93430
 Villetaneuse, France}

\email{queguin@math.univ-paris13.fr}

\thanks{Thanks to R. Parimala  for many fruitful discussions on the subject and for her support while were preparing this work. Thanks also to Eva Bayer, Skip Garibaldi, Nikita Karpenko and Jean-Pierre Tignol for helpful comments on an earlier draft of this work. The second-named author acknowledges the support of the French Agence Nationale de la Recherche (ANR) under reference ANR-12-BL01-0005}


\begin{abstract}
Let $q$ be a quadratic form over a field $k$ and let $L$ be a field extension of $k$ of odd degree. It is a classical result that if $q_L$ is isotropic (resp. hyperbolic) then $q$ is isotropic (resp. hyperbolic). In turn, given two quadratic forms $q, q^\prime$ over $k$, if $q_L \cong q^\prime_L$ then $q \cong q^\prime$. It is natural to ask whether similar results hold for algebras with involution. We give a survey of the progress on these three questions with particular attention to the relevance of hyperbolicity, isotropy and isomorphism over some {appropriate} function field. Incidentally, we prove the anisotropy property in some {new} low degree cases.
 \end{abstract}

\maketitle

\section*{Introduction}
\label{intro.sec}

Let $F$ be a field of characteristic different from 2. It is well-known that an anisotropic quadratic form $q$ over $F$ is anisotropic over any finite field extension of $F$ of odd degree. This result was first published by T.A. Springer \cite{Springer} in 1952, but Emil Artin had already communicated a proof to Witt by 1937~{see~\cite[Remark 1.5.3]{Kahn}}. In what follows, we refer to this result as \emph{the Artin-Springer theorem}. Since any quadratic form can be decomposed as the sum of an anisotropic part and some number of hyperbolic planes, an immediate consequence of the Artin-Springer theorem is that a quadratic form which becomes hyperbolic over an odd-degree field extension is hyperbolic. Further, since two quadratic forms $q$ and $q^\prime$ are isomorphic if and only if $q \perp-q^\prime$ is hyperbolic, another consequence of the Artin-Springer theorem is that two quadratic forms which become isomorphic over an odd-degree field extension are isomorphic. This last result also extends to similar quadratic forms. Indeed, using the properties of Scharlau transfer's map described in~\cite[Chap. 2, Thm. 5.6, Lem. 5.8]{Scharlau}, one may check that two forms which become similar after an odd degree field extension are similar.

Recall that every (nondegenerate) quadratic form $q$ on an $F$-vector space $V$ induces the so-called adjoint involution $\ad_q$ on the endomorphism algebra $\End_F(V)$, and, conversely, every orthogonal involution on $\End_F(V)$ is adjoint to a quadratic form $q$, uniquely defined up to a scalar factor. Therefore, algebras with orthogonal involution can be thought of as twisted forms (in the sense of Galois cohomology) of quadratic forms up to scalars. {Since $\ad_q$ is isotropic (resp. hyperbolic) if and only if $q$ is isotropic (resp. hyperbolic) and $\ad_q$ is isomorphic to $\ad_{q^\prime}$ if and only if $q$ and $q^\prime$ are similar,} it is natural to ask whether the behavior of quadratic forms under odd-degree field extensions described above, extends to involutions on central simple algebras. The present paper is mostly a survey of what is known on this topic. More precisely, we are interested in the following:

\begin{description}
\item[Main questions]
Let $F$ be a field and let $(A,\sigma)$ be an algebra with involution over $F$. Let $L$ be an odd-degree field extension of $F$.

\begin{enumerate}
\item[(i)] If $\sigma$ is anisotropic, does it remain anisotropic over $L$? 
\item[(ii)] If $\sigma$ is non-hyperbolic, does it remain non-hyperbolic over $L$? 
\item[(iii)] If $\sigma$ and $\sigma^\prime$ are non-isomorphic involutions, do they remain non-isomorphic over $L$? 
\end{enumerate}
\end{description}

Question (ii) was solved by Bayer and Lenstra~\cite{BayerLenstra}, in an even more general context than is discussed above, see~\S~\ref{bl.sec} below. Question (i) should be posed differently, as was noticed by {Parimala, Sridharan and Suresh. In~\cite[\S 4]{PSS}, they constructed an example of an anisotropic unitary involution that becomes isotropic over an odd-degree field extension. They suggested the following reformulation:
\begin{enumerate}
\item[(i$'$)]
Let $(A,\sigma)$ be an algebra with involution over $F$, and let $L/F$ be a field extension of degree coprime to $2\ind(A)$. 
If $\sigma$ is anisotropic, does it remain anisotropic over $L$\footnote{More generally, one can ask how the Tits index of an algebraic group behaves over finite field extensions of degree coprime to the torsion primes of the group. See~\cite[Problem 7.3]{ABGV}}? 
\end{enumerate} 
Questions (i)  and (i$'$) are equivalent if the involution is orthogonal or symplectic, {since an algebra which admits an involution of either of these types has exponent 2 and the index and exponent of any central simple algebra have the same prime factors. By similar reasoning, the two questions} are equivalent in the unitary case under the additional hypothesis that the algebra has $2$-power exponent.} 

Question  (i$'$) is open in general, {though as we will discuss in ~\S~\ref{hypiso.sec},~\ref{pss.sec},~\ref{generic.sec} and~\ref{lowdeg.sec} a positive answer is known for algebras with involution satisfying some additional conditions.} By the {aforementioned} Bayer-Lenstra theorem, question  (i$'$) has a positive answer for involutions for which isotropy and hyperbolicity are equivalent. In particular, a positive answer is known for totally decomposable involutions, by results of Becher~\cite{Becher} and Karpenko~\cite{Karphyporth}; this is explained in \S~\ref{hypiso.sec}. Parimala, Sridharan and Suresh gave a general argument for algebras of index $2$ with orthogonal involution, based on the excellence property of the function field of a conic~\cite{PSS}, see~\S~\ref{pss.sec}. In~\S~\ref{lowdeg.sec} we prove new results on low-degree algebras, in particular, degree $12$ algebras with orthogonal involutions, thus answering a question posed in~\cite[pg 240]{ABGV}.
This new case includes some algebras of index strictly larger than $2$, and for which isotropy is not equivalent to hyperbolicity, so that the question does not reduce to the Bayer-Lenstra theorem. 

A natural way to address question  (i$'$) is to try to reduce to quadratic form theory by extending scalars to a function field\footnote{This approach also relates our main question to the following classical question for algebraic groups: Let $F$ be a field, $G$ and $G'$ be algebraic groups over $F$, and $X$ and $X'$ projective homogeneous varieties under $G$ and $G'$, respectively. When does $X$ admit a rational point over the function field $F(X^\prime)$? See for instance~\cite{Kahn} for results for quadratic forms and the so-called index reduction formulas (e.g. \cite{MPW}) for results for central simple algebras.}. {This method was used more than a decade ago by Parimala-Sridharan-Suresh~\cite{PSS}, Dejaiffe~\cite{Dejaiffe} and Karpenko~\cite{Karpanisotropy} to study isotropy of orthogonal involutions.}
{Roughly speaking, one uses the existence of generic index reduction fields $\cF_{A,t}$ depending { on the algebra $A$ and }on the type of the involution, over which $\sigma$ is adjoint to a hermitian form, which in turn is determined by an associated quadratic form.  In the orthogonal case, one may take a generic splitting field of the algebra $A$, since the involution is adjoint to a quadratic form over such a field; see \S~\ref{functionfield.notation} below for a description of $\cF_{A,t}$ in the symplectic and unitary cases. {If one can prove that an anisotropic involution of type $t$ remains anisotropic over $\cF_{A,t}$, then a positive answer to question (i$'$) (and even question (i)) follows easily from the Artin-Springer theorem} (see Lemma~\ref{as.lem}). On the other hand, it is a deep result, due to Karpenko~\cite{Karpisoorth}, Tignol~\cite[Apppendix]{Karpisoorth} and Karpenko-Zhykhovich~\cite{KZ} that if anisotropy is preserved under odd-degree field extensions, then it is preserved under extension to $\cF_{A,t}$. Therefore, question (i) is equivalent to asking whether anisotropy is preserved over $\cF_{A,t}$ (for algebras of $2$-power index in the unitary case). 
A survey on this approach is the content of~\S~\ref{generic.sec}, where we also explain how one can reduce question (i) to an excellence question. }

An affirmative answer to question (iii) for symplectic and orthogonal involutions was given by Lewis~\cite[Proposition 10]{Lewis} and Barquero-Salavert~\cite[Theorem 3.2]{Barquero} proved an affirmative answer for unitary involutions (See also \cite[Proposition 5.1]{Black1}). The second-named author and Tignol~\cite[\S 4]{QT}, produced examples of non-isomorphic orthogonal involutions that become isomorphic after generic splitting of the underlying algebra. In particular, the behavior of non-isomorphic involutions is not the same under finite odd-degree extensions and extension to $\cF_{A,t}$, see~\S~\ref{isom.sec}.

\section{Background and Notation}
\label{background.sec}

In this section, we review the relevant background that informs this work. The results on quadratic forms mentioned in the introduction are explained in~\cite{EKM},~\cite{Kahn},~\cite{Lam}, and~\cite{Scharlau}, while general facts on algebras with involution and hermitian forms are in~\cite{KMRT}. 

{Throughout the paper, $A$ denotes a central simple algebra over a field $K$ of {characteristic}\footnote{{This restriction on the characteristic of $K$ is not always necessary. For instance, the Artin-Springer theorem is valid in characteristic $2$~\cite[18.5]{EKM}. The main result in~\cite{Karpisopair}, which will be discussed in \S~\ref{generic.sec} below, holds over a field of arbitrary characteristic. However, as most of the results which inform this survey are for fields of characteristic different from $2$, we observe that convention.}}} different from 2. An \emph{involution} $\sigma$ on $A$ is an anti-automorphism of period 2. The involution is said to be of \emph{orthogonal type, symplectic type} or  \emph{unitary type}, according to the type of its automorphism group. We consider as a base field the field $F$ of elements of $K$ fixed by $\sigma$. If $\sigma$ is unitary, $K/F$ is a quadratic field extension. Otherwise, $K=F$ and $\sigma$ is $K$-linear. In all three cases, we say for short that $(A,\sigma)$ is \emph{an algebra with involution over $F$}. 
{Two $F$-algebras with involution $(A,\sigma)$ and $(A^\prime,\sigma^\prime)$ are \emph{isomorphic} if there is an $F$-algebra isomorphism $f:A \to A^\prime$ such that $f \circ \sigma=\sigma^\prime \circ f$. Since in the unitary case, the isomorphism $f$ induces an isomorphism of the centers of the algebras $K$ and $K'$, we may assume that $K=K'$ and that $f$ is $K$-linear.}

By Wedderburn's theorem, the algebra $A$ can always be represented as an endomorphism algebra $A\simeq \End_D(V)$, where $D$ is a central division algebra Brauer equivalent to $A$, $V$ is a $D$-module, and both are uniquely defined up to isomorphism. The degree of $D$ is called the \emph{index} of $A$, and we call the dimension  of $V$ over $D$, the {\em co-index} of $A$. Thus the degree of $A$ is the product of its index and its co-index. It follows from the existence criteria for involutions~\cite[\S 3]{KMRT} that $D$ is endowed with an involution $\theta$ of the same type as $\sigma$. Once such a $\theta$ is chosen, $\sigma$ can be represented as the adjoint involution with respect to a hermitian form $h$ over $(D,\theta)$, which is uniquely defined up to a scalar factor. We will refer to such a form $h$ as {\em a hermitian form associated to} $\sigma$. 
{For any field extension $L/F$, we denote by $(A_L,\sigma_L)$ the \emph{extended algebra with involution}, defined by $A_L=A\otimes_F L$ and $\sigma_L=\sigma\otimes \id$.  Since an involution of any type acts on $F$ as $\id_F$, $\sigma_L$ is well-defined. Given a representation $(A,\sigma)=(\End_D(V),\ad_h)$, for some hermitian module $(V, h)$ over $(D,\theta)$ we denote by $V_L$ the $D_L$-module $V_L=V\otimes_F L$ and by $h_L$ the extended form $h_L:\, V_L\times V_L\rightarrow (D_L,\theta_L)$, so that $(A_L,\sigma_L)\simeq (\End_{D_L}(V_L),\ad_{h_L})$. }

{One may easily check that if some hermitian form associated with $\sigma$ is isotropic 
(resp. hyperbolic), so is every hermitian form associated with $\sigma$.} The involution is said to be \emph{isotropic} or \emph{hyperbolic} accordingly. A right ideal $I \subset A$ is called an \emph{isotropic ideal} if $\sigma(x)x=0$ for all $x\in I$. Given a representation $(A,\sigma)\simeq(\End_D(V),\ad_h)$, isotropic right ideals are given by endomorphisms of $V$ with image contained in a given totally isotropic $D$-subspace $W$ of $V$. The {\em reduced dimension} of such an ideal is the product of the index of $A$ and the dimension of $W$ over $D$. 
Hence, one may also use isotropic ideals to give a statement of the definition of isotropy and hyperbolicity of involutions, independent of the choice of a representation in terms of a hermitian module. Namely, $(A,\sigma)$ is \emph{isotropic} if and only if $A$ contains a nonzero isotropic ideal and $(A,\sigma)$ is \emph{hyperbolic} if and only if $A$ contains an isotropic ideal of reduced dimension $\frac 12 \deg(A)$ (see~\cite{BST} or~\cite[\S6]{KMRT}). In particular, hyperbolic involutions can only exist on algebras of even co-index. 

{{Given an algebra $A$ with an involution $\sigma$ of any type $t$, we define below a field $\cF_{A,t}$\label{functionfield.notation} such that the involution $\sigma_{\cF_{A,t}}$ is either adjoint to a quadratic form or adjoint to a hermitian form determined by a quadratic form. In particular, even in the case of a unitary or symplectic involution, there is a quadratic form $q$ over $\cF_{A,t}$ such that  $\sigma_{\cF_{A,t}}$ is isotropic (resp. hyperbolic) if and only if $q$ is isotropic (resp. hyperbolic).} The field $\cF_{A,t}$ depends on the algebra $A$ and on the type $t$ of $\sigma$ and is defined as follows. We set $t=o$ (respectively $s, u$) when $\sigma$ is of orthogonal (respectively symplectic, unitary) type.} Assume the involution $\sigma$ is of orthogonal type. We let $\cF_{A,o}$ be the function field of the Severi-Brauer variety ${\mathrm {SB}}(A)$ of $A$. Since $A$ is split over $\cF_{A,o}$ the involution $\sigma_{\cF_{A,o}}$ is adjoint to a quadratic form. 
If $\sigma$ is unitary, we need an extension of the fixed field $F=K^\sigma$ to extend the involution. Thus, we consider the function field $\cF_{A,u}$ of the Weil transfer of the Severi-Brauer variety of $A$. Since $A$ is split over $\cF_{A,u}$ the extended involution $\sigma_{\cF_{A,u}}$ is adjoint to a hermitian form $h$ with values in the quadratic extension $\cF_{A,u}\otimes_F K$. We can associate to $h$ the quadratic form $q_h:\,V\rightarrow \cF_{A,u}$ defined by $q_h(x):=h(x,x)$. It is classically known that the hermitian form $h$ is uniquely determined by $q_h$ and $q_h$ is called the \emph{trace form} of $h$. Moreover, the isotropy or hyperbolicity of $h$ is determined by that of $q_h$~\cite[Chap. 10, Thm. 1.1]{Scharlau}. A symplectic involution on a split algebra is hyperbolic. So, rather than considering a splitting field of $A$, we let $\cF_{A,s}$ be the function field  of the generalized Severi-Brauer variety ${\mathrm {SB}}_2(A)$ of right ideals of reduced dimension $2$. This field generically reduces the index of $A$ to 2. Further, given $H$ a quaternion algebra over $\cF_{A,s}$ Brauer equivalent to $A_ {\cF_{A,s}}$ and taking the canonical involution $\bar{~}$ on $H$, the involution $\sigma_{\cF_{A,s}}$ is adjoint to a hermitian form over $(H,\bar{~})$ determined as in the unitary case by its trace form~\cite[Chap. 10, Thm. 1.7]{Scharlau}. 

We will make frequent use of the following straightforward consequence of the Artin-Springer theorem:  
\begin{lemma}
\label{as.lem}
Let $(A,\sigma)$ be an algebra with involution of type $t$ over $F$. If $\sigma$ is non-hyperbolic (resp. anisotropic) over the function field $\cF_{A,t}$, {then it is non-hyperbolic (resp. anisotropic)} over any odd degree extension $L$ of the base field $F$. 
\end{lemma}
\begin{proof}
Assume $\sigma$ is hyperbolic (resp. isotropic) over an odd degree extension $L$ of the base field $F$. Then $\sigma$ is hyperbolic (resp. isotropic) over the compositum $\cL=\cF_{A,t}L$, which is an odd degree extension of $\cF_{A,t}$. Since the involution $\sigma$ is determined by a quadratic form over $\cF_{A,t}$ and over $\cL$, we can apply the Artin-Springer theorem to deduce that $\sigma_{\cF_{A,t}}$ is hyperbolic (resp. isotropic). 
\end{proof}

\section{Hyperbolicity of involutions}
\label{bl.sec}
\noindent
We begin by considering question (ii) above. As we mentioned, a complete answer to this question was given in 1990 by Bayer and Lenstra \cite{BayerLenstra}, who aimed at proving the existence of a self-dual normal basis for any odd-degree Galois field extension. Their argument is based on the following, which is the result we are interested in: 
\begin{theorem}\cite[Proposition 1.2]{BayerLenstra} 
\label{BL}
Let $B$ be a finite dimensional $F$-algebra endowed with an $F$-linear involution $\theta$, and $(V,h)$ a hermitian module over $(B,\theta)$. Let $L$ be a field extension of $F$ of odd degree. If $(V_L,h_L)$ is hyperbolic, then $(V,h)$ is hyperbolic.
\end{theorem}
\noindent
In view of the definition of hyperbolic involutions via the associated hermitian forms, this result gives a positive solution to question (ii). Since a $K$-central division algebra with $K/F$ unitary involution is an $F$-algebra with $F$-linear involution, Theorem~\ref{BL} applies to involutions of any type. Further, since there is no simplicity assumption on the algebra $B$, Theorem~\ref{BL} applies to a broader class of algebras with involution than is specified in the statement of question (ii). The proof is quite similar to the classical proof in quadratic form theory and is based on Scharlau's transfer homomorphism, which, as the authors prove, extends naturally to the setting of hermitian forms.

In the sequel, we will frequently use the following corollary of Bayer-Lenstra's theorem: 
\begin{corollary}
\label{bl.cor}
\cite[Corollary 1.4]{BayerLenstra}
Let $(V,h)$ and $(V^\prime,h^\prime)$ be two hermitian forms over $(B,\theta)$, and let $L$ be an extension of $F$ of odd degree. If the extended forms $(V_L,h_L)$ and $(V^\prime_L, h^\prime_L)$ are isomorphic, then $(V,h)$ and $(V^\prime,h^\prime)$ are isomorphic. 
\end{corollary}
Note that this result does not answer the isomorphism question for involutions, since $\sigma \simeq \sigma^\prime$ implies only that their associated hermitian forms are similar. However, as we describe in \S~\ref{isom.sec}, one can use Scharlau's norm principle to deduce a positive answer to the isomorphism question. 

{Though Theorem~\ref{BL} gives a very nice and purely algebraic solution to question (ii), it is natural} to ask whether non-hyperbolicity is preserved under scalar extension to the function field $\cF_{A, t}$ (see \S~\ref{background.sec} for the definition of $\cF_{A,t}$, depending on the type $t$ of $\sigma$). The following result is due to Karpenko for orthogonal and unitary involutions, and Tignol for symplectic involutions\footnote{Tignol's argument also applies to the unitary case if the underlying  algebra has exponent $2$.}; it was previously proven by Dejaiffe~\cite{Dejaiffe} and Parimala, Sridharan and Suresh~\cite{PSS} for algebras of index $2$ with orthogonal involutions: 
\begin{theorem}\cite[Theorems 1.1 \& A.1]{Karphyporth}\cite[Theorem 1.1]{Karphypuni}
\label{Kh} 
Let $(A,\sigma)$ be an algebra with involution of type $t$ over $F$. If the extended involution $\sigma_{\cF_{A,t}}$ is hyperbolic, then $\sigma$ is hyperbolic.
\end{theorem}


\label{genericodd.desc}
Using this result and the Artin-Springer theorem for quadratic forms (or even its weak version, for hyperbolicity) one gets another argument for a positive answer to question (ii) (see Lemma~\ref{as.lem}). Therefore, Theorem~\ref{Kh} can be considered as a generalization {of Bayer-Lenstra's Theorem~\ref{BL}} in the setting of central simple algebras. Its proof, based on computations of cycles on the underlying varieties, requires much more machinery than the original proof of Bayer and Lenstra, including {for instance} the Steenrod operations on Chow groups with coefficients in ${\mathbb Z}/2$. {Yet}, in addition to generalizing this previous result, Theorem~\ref{Kh} can be seen as an intermediate result that helped to pave the way for Karpenko and {Karpenko-}Zhykovich's later results on isotropy~\cite{Karpisoorth},~\cite{KZ}. {Moreover}, Theorem~\ref{Kh} also implies an interesting connection between isotropy and hyperbolicity of totally decomposable involutions {which we shall discuss in the next section.} 

\section{First results on isotropy}
\label{hypiso.sec}

Though question (i$'$) is largely open, there is some evidence in support of an affirmative answer. In the case where $A$ is split and $\sigma$ is orthogonal or unitary, or $A$ has index $2$ and $\sigma$ is symplectic, one can reduce to quadratic form theory (see \S~\ref{background.sec}) where the Artin-Springer theorem shows that anisotropy is preserved under any odd degree extensions. Since every symplectic involution on a split algebra is hyperbolic, this resolves the split case, for involutions of all three types. For division algebras, an affirmative answer to question (i$'$) is a consequence of the fact that, since $[L:F]$ is coprime to the index of $A$, the extended algebra $A_L$ is division. {Since a division algebra admits no isotropic involution, $\sigma_L$ remains anisotropic.}

{Next we highlight two cases in which isotropy reduces to hyperbolicity,  so that a positive answer to question (i$'$) follows from results in \S~\ref{bl.sec} above. The first such case is when $A$ has co-index 2. In this setting, the involution $\sigma$ is adjoint to a 2-dimensional hermitian form. Since isotropy is equivalent to hyperbolicity for such a form, Theorem~\ref{BL} {gives} that anisotropy is preserved {under field extensions of degree coprime to $2\ind(A)$}.

One can make a similar reduction for totally decomposable involutions. The following result, {due to Becher~\cite{Becher} in the orthogonal and symplectic cases}, will prove useful for this purpose:
\begin{theorem}\cite[Theorem 1 \& Corollary]{Becher} \label{Becher}
Let $(A,\sigma)$ be an algebra with involution over $F$ that decomposes as a tensor product of quaternion algebras with involution. We assume moreover that $A$ is split if $\sigma$ is orthogonal or unitary, and has index $2$ if $\sigma$ is symplectic. Then there exists a Pfister form $\pi$ such that $(A,\sigma)$ decomposes as follows: 
\begin{enumerate} 
\item $(A,\sigma)=\Ad_\pi$ if $\sigma$ is orthogonal; 
\item $(A,\sigma)=\Ad_\pi\otimes_F (K,\bar{\ })$ if $\sigma$ is unitary; 
\item $(A,\sigma)=\Ad_\pi\otimes_F (H,\bar{\ })$ if $\sigma$ is symplectic. 
\end{enumerate}
\end{theorem} 
\begin{proof}
Assertions (1) and (3) are Theorem 1 and its Corollary in \cite{Becher}. Assertion (2) also follows easily from those results. Indeed, assume $(A,\sigma)$ is a tensor product of quaternion algebras, each endowed with a $K/F$ unitary involution. By~\cite[(2.22)]{KMRT}, each factor decomposes as $(H_i,\bar{\ })\otimes_F(K,\bar{\ })$, for some quaternion algebra $H_i$ over $F$. Therefore, $(A,\sigma)$ has a decomposition $(A,\sigma)=(A_0,\sigma_0)\otimes _F(K,\bar{\ })$, where $(A_0,\sigma_0)$ is a totally decomposable algebra with orthogonal or symplectic involution, depending on the parity of the number of factors. Moreover, since $A=A_0\otimes_F K$ is split, $A_0$ has index at most $2$. Therefore, by~\cite[Corollary \& Theorem 2]{Becher}, $(A_0,\sigma_0)=\Ad_{\pi_0}\otimes (H_0,\gamma)$ for some Pfister form $\pi_0$ and some quaternion algebra with orthogonal or symplectic involution $(H_0,\gamma)$. To conclude, it only remains to observe that since $H_0\otimes_F K$ is split, $\gamma\otimes \bar{\ }$ is adjoint to a $2$-dimensional hermitian form with values in $(K,\bar{\ })$. Up to a scalar, this form has a diagonalisation $\qform{1,-\mu}$ for some $\mu\in F^\times$, so that $(H_0,\gamma)\otimes_F(K,\bar{\ })\simeq \Ad_\qform{1,-\mu}\otimes_F (K,\bar{\ })$. This concludes the proof, with $\pi=\pi_0\otimes\qform{1,-\mu}$. 
 \end{proof}
 
For totally decomposable involutions, we get an affirmative answer to question (i$'$), and even a slightly stronger result in the unitary case. Indeed, given Becher's Theorem~\ref{Becher}, {Karpenko's Theorem}~\ref{Kh} admits the following corollary:  \begin{corollary}
\label{totdec.cor}
Let $(A,\sigma)$ be an algebra with involution over $F$ that decomposes as a tensor product of quaternion algebras with involution over $F$. 
If $\sigma$ is isotropic, then it is hyperbolic. In particular, if $\sigma$ is anisotropic, it remains anisotropic over any odd degree extension of the base field.
\end{corollary}
\begin{proof} 
If $\sigma$ is isotropic, then $\sigma_{\cF_{A,t}}$ is isotropic. Therefore, in view of~\ref{Kh}, it is enough to prove that isotropy implies hyperbolicity over $\cF_{A,t}$, or equivalently, that isotropy implies hyperbolicity for split algebras with orthogonal or unitary involution and index $2$ algebras with symplectic involution. {This was shown by Becher~\cite{Becher} in the orthogonal and symplectic cases, and follows easily from~\ref{Becher}. In the orthogonal case, $\sigma$ is adjoint to a Pfister form $\pi$. In the unitary and symplectic cases, $\sigma$ is adjoint to a hermitian form with trace form $q_h=\pi\otimes \qform{1,-\delta}$ and $q_h=\pi\otimes n_H$ respectively, where $\delta$ is given by $K=F(\sqrt\delta)$, and $n_H$ denotes the norm form of $H$.} In particular, in all three cases, the involution $\sigma$ is determined by a quadratic form which is a Pfister form. Since isotropy implies hyperbolicity for Pfister forms, we obtain the desired result. The last assertion follows by Bayer-Lenstra's Theorem~\ref{BL}. 
\end{proof} 
The first example of a positive answer to question (i$'$) that does not reduce to hyperbolicity is due to Parimala, Sridharan and Suresh~\cite{PSS}. That result is the main subject of the next section.

 \section{Orthogonal involutions on algebras of index $2$ and excellence for hermitian forms}
 \label{pss.sec}
 
{An affirmative answer to question (i$'$) for algebras of index 2 with orthogonal or symplectic involution was proven by Parimala, Sridharan and Suresh \cite{PSS}. The argument in the symplectic case is elementary and was explained at the beginning of the previous section (see also~\cite[Proof of thm 3.5]{PSS}). The key result in the orthogonal case is an excellence result for function fields of conics. We recall that a field extension $L/F$ is said to be \emph{excellent} for a hermitian form $h$ {over $F$} if there is a hermitian form $h^\prime$ {over F} such that $h^\prime_L$ is isomorphic to the anisotropic part of $h_L$. {The authors prove excellence of the function field of any smooth projective conic defined over the base field.}

\begin{theorem} \cite[Theorem 2.2]{PSS} \label{excellence} Let $(D,\theta)$ be an algebra with orthogonal or symplectic involution over $F$ and let ${\mathcal C}$ be a smooth, projective conic over $F$, with function field $\cF$. The  field extension $\cF/F$ is excellent for hermitian forms over $(D,\theta)$. \end{theorem} 

{Since when $A$ has index 2, the field $\cF_{A,o}$ is a purely transcendental extension of the function field of a conic, with Theorem~\ref{excellence} in hand, one can show the following:}

\begin{corollary} \cite[Cor. 3.4, thm. 3.5]{PSS} \label{PSSgeneric}
Let $A$ be an algebra of index 2, and let $\cF_{A,o}$ be the function field of the Severi-Brauer variety of $A$. Anisotropic orthogonal involutions on $A$ remain anisotropic over the function field $\cF_{A,o}$, hence also over all odd-degree field extensions of $F$. 
\end{corollary}
The argument in~\cite{PSS} goes as follows. Let $Q$ be a quaternion division algebra Brauer equivalent to $A$.Pick an orthogonal involution $\theta$ on $Q$ so that $\sigma$ is adjoint to some hermitian form $h$ over $(Q,\theta)$. 
If $\sigma$ is isotropic over $\cF_{A,o}$, then $h$ is isotropic over $\cF_{A,o}$. Since $A$ has index 2, the field $\cF_{A,o}$ is a purely transcendental extension of the function field of a conic and thus by excellence, there is a hermitian form $h^\prime$ over $(Q,\theta)$ such that $h^\prime_{\cF_{A,o}}$ is isomorphic to the anisotropic part of $h_{\cF_{A,o}}$. Hence, the form $h\perp -h'$ is hyperbolic over $\cF_{A,o}$ and by Theorem~\ref{Kh} (see also \cite[Prop]{Dejaiffe},\cite[Proposition 3.3]{PSS}), it follows that $h\perp-h'$ is hyperbolic over $F$. In view of the dimensions, this implies that $h$ is isotropic.}
The second assertion follows immediately, by the Artin-Springer theorem (see Lemma~\ref{as.lem}). 
Hence this gives a positive answer to question (i$'$) for orthogonal involutions on algebras of index $2$.

The excellence property of function fields of conics used in this section does not extend to function fields of Severi-Brauer varieties. For instance, Izhboldin-Karpenko~\cite{IK} and Sivatski~\cite{Sivatski2} proved that the function field of a division biquaternion algebra {does not satisfy the excellence property} for quadratic forms over the base field. {However}, the proof of Corollary~\ref{PSSgeneric} only uses a very particular case of the excellence property, namely excellence for hermitian forms over the underlying quaternion algebra with orthogonal involution $(Q,\theta)$. 
{This part of the argument is general, and one can reduce the anisotropy question to an excellence {question, since anisotropy is known to be preserved over $\cF_{D,t}$ for any involution on a division algebra $D$}. This latter result is due to Karpenko and Tignol and is discussed in the next section.}

\section{Isotropy over the function field $\cF_{A,t}$}
 \label{generic.sec}
 
In his first paper on the topic of the present survey~\cite{Karpanisotropy}, Karpenko stated the following conjecture:
\begin{conjecture}\label{generic}
\cite[Conjecture 5.2]{Karpanisotropy} Let $A$ be a central simple algebra over a field $F$. An anisotropic orthogonal involution remains anisotropic over $\cF_{A,o}$.  
\end{conjecture}

{The analogous statement for unitary involutions is false in general. Parimala, Sridharan and Suresh~\cite[Thm 4.3]{PSS} showed that for every odd prime $p$, there is an algebra of $p$-power index with unitary involution which becomes isotropic over an odd-degree field extension. By Lemma~\ref{as.lem}, it follows that the involution also becomes isotropic over $\cF_{A,t}$. However, whether the conjecture holds not just for orthogonal involutions, but also for symplectic involutions or unitary involutions on algebras of $2$-power index is an open question. There is some evidence in support of the conjecture in these settings.} 

The conjecture holds if $A$ is split and $\sigma$ is orthogonal since in this case $\cF_{A,o}$ is a purely transcendental extension of $F$. The same argument applies to the unitary split case. Since there are no anisotropic symplectic involutions on a split algebra, the result also trivially holds in the split symplectic case. For division algebras, Karpenko and Tignol proved the following: 

\begin{theorem}\cite[Theorem 5.3]{Karpanisotropy}, \cite[Theorem. A.1 \& A.2]{Karphyporth}
\label{division.thm}
{Let $(D,\theta)$ be a division algebra with an anisotropic involution of type $t$. If $\theta$ is unitary, we assume that $D$ has exponent $2$. The involution} $\theta$ remains anisotropic over $\cF_{D,t}$.
\end{theorem}

In the orthogonal case, Theorem ~\ref{division.thm} is the main result in~\cite{Karpanisotropy}. Its proof {involves} cycle computations in Chow groups with values in ${\mathbb Z}/2$. Tignol extended the {result} to symplectic {involutions}, and unitary involutions on algebras of exponent $2$. {We sketch {here} {the} beautiful} and rather elementary argument {that he gives} {in the symplectic case}. Given an algebra $D$ with symplectic involution $\theta$ over $F$, consider the iterated Laurent series field $\hat F=F((x))((y))$, the quaternion division algebra $(x,y)_{\hat F}$ over $\hat F$, and the tensor product $\hat D=D\otimes_F (x,y)_{\hat F}$. Since the canonical involution $\gamma$ on $(x,y)_{\hat F}$ is symplectic, the involution $\hat\theta=\theta\otimes \gamma$ on $\hat D$ is orthogonal. Using some residue computations, one may check that $D$ is division if and only if $\hat D$ is division, and that $\theta$ is anisotropic if and only if $\hat\theta$ is anisotropic. {Since $\theta$ is anisotropic by assumption}, Karpenko's result in the orthogonal case implies that $\hat\sigma$ remains anisotropic over $\cF_{\hat D,o }$. Further, since $\hat D$ necessarily splits over  $\cF_{\hat D, o}$ then $A$ is Brauer equivalent to $(x,y)_{\cF_{\hat D, o}}$ over $\cF_{\hat D,o }$. Thus $D$ has index $2$ {and $\theta$ is anisotropic} over this field. It follows that $\theta$ remains anisotropic over the generic index reduction field $\cF_{D,s}$.

Recall that by Lemma~\ref{as.lem}, a proof of Conjecture~\ref{generic} would give an affirmative answer to question (i$^\prime$). In fact, {interest in this question seemed to be the initial motivation} for studying Conjecture~\ref{generic}. However, within a decade of stating the conjecture, Karpenko, Tignol and Karpenko-Zhykhovich had proven the following converse of Lemma~\ref{as.lem}: 

\begin{theorem}\cite[Theorem 1]{Karpisoorth}\cite[Appendix]{Karpisoorth},\cite{Karpisosym},\cite[Theorem 6.1]{KZ}
\label{karmainthm}
Let $(A,\sigma)$ be a central simple algebra over $F$ with an anisotropic involution of type $t$. If $\sigma_{\cF_{A,t}}$ is isotropic, then there exists an odd-degree field extension $L/F$ such that $\sigma_L$ is isotropic. 
\end{theorem} 

Theorem~\ref{karmainthm} is a very deep result whose proof is related to the incompressibility of some projective homogeneous varieties. Among its interesting consequences is the following extension of Theorem~\ref{division.thm} for unitary involutions, {which actually applies to division algebras of arbitrary $2$-power exponent:
\begin{corollary}\cite[Thm. 1.4]{Karphypuni}\footnote{In fact,~\cite[Thm. 1.4]{Karphypuni} is more general than the consequence we point out here, and was proven before Theorem~\ref{karmainthm}, see~\ref{index.cor} below. The corollary can also be deduced from the earlier result of N. Karpenko~\cite{Karpincomp}.}\label{divgen.thm}
Let $D$ be a division algebra with $2$-power exponent. Any anisotropic unitary involution on $D$ remains anisotropic over $\cF_{D,u}$. 
\end{corollary}
\begin{proof}
By Theorem~\ref{karmainthm}, isotropy over $\cF_{D,u}$ would imply isotropy over some odd degree extension $L$ of the base field. But this is impossible, since the hypothesis on the exponent of $D$ guarantees that $D$ remains division over $L$.
\end{proof}

A further consequence of Theorem~\ref{karmainthm} is the following:
\begin{corollary}
\label{excellence.prop}
{Let $(A,\sigma)$ be a central simple algebra over $F$ with anisotropic involution of type $t$.
Fix an involution $\theta$ of type $t$ on the underlying division algebra $D$ and let $h$ over $(D,\theta)$ be a hermitian form associated to $\sigma$.}
The following assertions are equivalent: 
\begin{enumerate}
\item[(i)] The extension $\cF_{D,t}/F$ is excellent for the hermitian form $h$;
\item[(ii)] The involution $\sigma$ remains anisotropic over $\cF_{D,t}$;
\item[(iii)] The involution $\sigma$ remains anisotropic over any odd degree field extension $L/F$;
\end{enumerate}
\end{corollary} 
The equivalence between conditions (ii) and (iii) is given by Lemma~\ref{as.lem} and Theorem~\ref{karmainthm}. 
Condition (ii) clearly implies condition (i), and conversely, condition (i) combined with Theorem~\ref{Kh} implies condition (ii), as in the proof of Parimala-Sridharan-Suresh's {result Theorem}~\ref{PSSgeneric}. 

{We conclude this section, by pointing out a third} interesting consequence of {Theorem}~\ref{karmainthm} which is further evidence in support of an affirmative answer to question (i$'$). {The statement of the result requires the notion of the Witt index of an involution}. We define the \emph{Witt-index} of an involution $\sigma$ to be the reduced dimension of a maximal isotropic ideal of the underlying algebra $A$.\footnote{The Witt index of $\sigma$ is the maximal element of the index of $(A,\sigma)$ as defined in~\cite[\S 6.A]{KMRT}.} If $\sigma$ is adjoint to a hermitian form $h$, then the Witt indices of $\sigma$ and $h$ {satisfy the equation} $i_W(\sigma)=i_W(h)\ind(A)$, where $\ind(A)$ is the Schur index of $A$ \cite[\S 6.A]{KMRT}. 
\begin{corollary} 
\label{index.cor}
{Let $(A,\sigma)$ be an algebra with involution of type $t$ over $F$. Then, there exists an odd degree field extension $L/F$ such that $i_W(\sigma_L)=i_W(\sigma_{\cF_{A,t}})$. If in addition $A$ has $2$-power exponent, then} $i_W(\sigma_{\cF_{A,t}})$ is a multiple of the Schur index of $A$.\footnote{For an algebra with orthogonal or unitary involution $(A,\sigma)$ over $F$, that the Witt index of $\sigma_{\cF_{A,o}}$ (resp. $\sigma_{\cF_{A,u}}$) is a multiple of the Schur index of $A$ is due to Karpenko~\cite[Thm 3.3]{Karpisopair},~\cite[Thm. 1.4]{Karphypuni}. This result preceded {Theorem}~\ref{karmainthm}, though as we have discussed here it can be deduced as a consequence of {Theorem} \ref{karmainthm}. The argument in \cite{Karpisopair} also applies to a quadratic pair over a base field of characteristic $2$~\cite[\S 5.B]{KMRT}.}
\end{corollary}
\begin{proof}
{Let $L$ be an odd degree extension of the base field $F$, and consider the compositum $\cL=L\cF_{A,t}$. Since the Witt index can only increase under a field extension, we have $i_W(\sigma_{L})\leq i_W(\sigma_{\cL})$. On the other hand, as in the proof of Lemma~\ref{as.lem}, the Artin-Springer theorem gives that $i_W(\sigma_{\cF_{A,t}})=i_W(\sigma_\cL)$. Therefore, $i_W(\sigma_L)\leq i_W(\sigma_{\cF_{A,t}})$.  If the inequality is strict, one may apply Theorem~\ref{karmainthm} to the anisotropic part of $(A,\sigma)_L$ to produce an odd degree extension $L'$ of $L$ such that $i_W(\sigma_{L'})>i_W(\sigma_L)$, and an induction argument completes the proof. The last assertion follows immediately since the Witt index of $\sigma_L$ is a multiple of $\ind(A_L)$ and our hypothesis on the exponent of $A$ guarantees that $\ind(A)=\ind(A_L)$.}
\end{proof}

\section{Isomorphism of involutions}
\label{isom.sec}

{While the main result in the previous section was that anisotropy is preserved under odd-degree extensions if and only if it is preserved under extension to $\cF_{A,t}$, we shall observe here that though non-isomorphic involutions remain non-isomorphic over odd degree field extensions, they may become isomorphic over $\cF_{A,t}$.} {Recall that Bayer and Lenstra's Corollary~\ref{bl.cor} gave that non-isomorphic hermitian forms remain non-isomorphic over odd-degree field extensions.} However, as we noted above, since isomorphic involutions have associated hermitian forms which are similar rather than isomorphic, this does not immediately answer question (iii). In 2000,  Lewis~\cite[Proposition 10]{Lewis} gave an affirmative answer to {question} (iii) for involutions of the first kind. Barquero-Salavert~\cite[Theorem 3.2]{Barquero} gave the proof in the unitary case in 2006 (see also \cite[Proposition 7.10, 7.20]{Blackdiss} and \cite[Theorem 4.8]{Black1}). The formal statement is as follows:
\begin{theorem}\cite[Proposition 10]{Lewis}, \cite[Theorem 3.2]{Barquero}. 
\label{isom.thm}{Let $A$ be a central simple algebra over $F$ and let $\sigma$ and $\sigma^\prime$ be two involutions on $A$ over $F$. 
Assume there is an odd-degree extension $L$ of $F$ such that $\sigma_L$ is isomorphic to $\sigma^\prime_L$. If the involutions are of unitary type, we assume in addition that the degree of $L/F$ is coprime to the index of $A$. Then $\sigma$ and $\sigma^\prime$ are isomorphic.}
\end{theorem}
The idea behind the proof is to use the extension of Scharlau's norm principle~\cite[Chap. 2 Thm 8.6]{Scharlau} to hermitian forms to descend the similarity factors of the hermitian forms associated to $\sigma$ and $\sigma^\prime$ over $L$ to the base field $F$. Since there is no norm map between $\cF_{A,t}$ and the base field $F$, there is no hope of using this strategy to produce an isomorphism over $F$ from a given isomorphism over $\cF_{A,t}$. Indeed, explicit examples of non-isomorphic orthogonal involutions that become isomorphic over $\cF_{A,o}$ are given by the second-named author and Tignol in~\cite{QT}. The construction can be sketched as follows:

The starting point is a field $F$ and a degree $8$ and exponent $2$ central simple algebra $E$ over $F$ that contains a triquadratic field $F(\sqrt a,\sqrt b,\sqrt c)$, but does not have any decomposition as a tensor product of quaternions $H_1\otimes H_2\otimes H_3$ such that $\sqrt a\in H_1$, $\sqrt b\in H_2$ and $\sqrt c\in H_3$. 
For instance, one may take for $E$ any indecomposable degree $8$ and exponent $2$ central simple algebra. Since the centralizer in $E$ of $K=F(\sqrt a)$ is a biquaternion algebra containing $F(\sqrt b,\sqrt c)$, {by Albert's theorem, it does decompose as $(b,r)\otimes (c,s)$} for some $r,s\in K^\times$. The example is defined in terms of these elements {$a$, $b$, $c\in F^\times$ and $s\in F(\sqrt a)^\times$. 
Pick two variables $x,y$ over $F$ and consider the algebra $A=(a,x)\otimes (b,y)\otimes (c,1)$ over $F(x,y)$. The element $s\in F(\sqrt a)^\times$ can be viewed as an element of $(a,x)\supset F(\sqrt a)$. Endow the quaternion algebras $(a,x)$, $(b,y)$ and $(c,1)$ with the unique orthogonal involutions $\rho$, $\tau$ and $\theta$ with discriminant $x$, $y$ and $c$ respectively (see~\cite[(7.4)]{KMRT}). Using cohomological invariants, it is proven in~\cite[Example 4.2]{QT} that the orthogonal involutions $\sigma=\rho\otimes\tau\otimes \theta$ and $\sigma'=({\mathrm{ int}}(s)\circ \rho)\otimes \tau\otimes \theta$ are non isomorphic, and become isomorphic {over $\cF_{A,o}$}. This example was inspired by Hoffmann's example~\cite[\S 4]{Hoffmann} of non-similar $8$-dimensional quadratic forms that are \emph{half-neighbors} {where two {$8$-dimensional} quadratic forms $q$ and $q^\prime$ over a field $F$ are said to be \emph{half-neighbors} if there exists {scalars $\lambda,\mu \in F$ such that $\qform{\lambda} q \oplus \qform{\mu} q^\prime$ is a $4$-fold Pfister form}. The relation between the two constructions is given by triality~\cite[\S 42]{KMRT}, and Sivatski's criterion for isomorphism after generic splitting of orthogonal involutions on degree $8$ algebras~\cite[Proposition 4]{Sivatski}. 

{In addition to being of independent interest, these results on isomorphism give new insights on isotropy. In the next section we use Lewis and Barquero-Salavert's Theorem \ref{isom.thm} to give an affirmative answer to question (i) for some algebras of low degree.}

\section{Isotropy of involutions on some algebras of low degree}
\label{lowdeg.sec}

{In~\cite[pg 240]{ABGV} Auel, Brussel, Garibaldi and Vishne note that it is unknown whether anisotropy is preserved under odd degree extensions for orthogonal involutions on algebras of degree $12$. We prove that this result holds and {thus} give a second example of a positive answer to question (i$'$) that does not reduce to hyperbolicity.}

\begin{theorem}
\label{deg12.prop}
Let $(A,\sigma)$ be a degree $12$ algebra with orthogonal involution over $F$. Let $L$ be a finite field extension of $F$ of odd degree. If $\sigma_L$ is isotropic, then $\sigma$ is isotropic. 
\end{theorem}
\begin{proof}
Since the cases where $\ind(A)=1$ or $\ind(A)=2$ were discussed above, we need only consider the case where $\ind(A)=4$. 
So, we may assume $A=M_3(D)$, for some division biquaternion algebra over $F$. Let us denote by $\delta\in\sq F$ the discriminant of $\sigma$, and let $F'=F[X]/(X^2-\delta)$ be the corresponding quadratic \'etale extension. 

Assume $\sigma$ is isotropic over $L$. Since $D$ remains division over $L$, the anisotropic part of $(A,\sigma)_L$ has degree $4$; therefore, it is isomorphic to $(D_L,\tilde\sigma)$, for some orthogonal involution $\tilde\sigma$ of $D_L$. 
The involution $\tilde\sigma$ being Witt-equivalent to $\sigma_L$, it has discriminant $\delta\in\sq F\subset\sq L$ and by the exceptional isomorphism~\cite[(15.7)]{KMRT}, 
we have 
$$(D_L,\tilde\sigma)=N_{L'/L}(\tilde Q,\bar{\ }),$$
where $L'=L[X]/(X^2-\delta)$, and the quaternion algebra $\tilde Q$ over $L'$ is the Clifford algebra of $\tilde\sigma$. 
Hence, since $L'$ has odd degree over $F'$, the Clifford algebra of $\sigma$ itself is Brauer-equivalent to a quaternion algebra $Q$ over $F'$ such that $Q_{L'}$ is isomorphic either to $\tilde Q$ or to its conjugate (see~\cite[prop. 3]{DLT}). In both cases, we have 
$\bigl(N_{F'/F}(Q,\bar{\ })\bigr)_L\simeq N_{L'/L}(\tilde Q,\bar{\ })$. 

To conclude, we will use the so-called fundamental relation~\cite[(9.14)]{KMRT}, which shows that $N_{F'/F}(Q)$ is isomorphic to the division algebra $D$. Therefore, $A=M_3(D)$ has a unique isotropic orthogonal involution $\sigma_0$ with anisotropic part $N_{F'/F}(Q,\bar{\ })$. 
Moreover, the involutions $\sigma$ and $\sigma_0$ are isomorphic over $L$. 
Hence they are isomorphic over $F$ by Lewis~\cite[Proposition.~10]{Lewis} (see~\ref{isom.thm} above), and $\sigma$ is isotropic. 
\end{proof}

{One may proceed as in Tignol's appendix~\cite[Appendix]{Karpisoorth} to derive a positive answer to question (i$'$) for an algebra of degree 6 and exponent 2 with unitary involution as a consequence of Theorem~\ref{deg12.prop}. Indeed, Tignol's construction associates to a degree $m$ algebra with $K/F$ unitary involution $(A,\sigma)$ a degree $2m$ algebra with orthogonal involution $(\tilde A,\tilde \sigma)$ over the Laurent series field $F((x))$, and he proves the anisotropy property holds for $(\tilde A,\tilde \sigma)$ if it holds for $(A,\sigma)$. Alternatively, one can prove this result by using an argument similar to that in the proof of Theorem \ref{deg12.prop}. We give the details of the latter approach below.}

\begin{theorem} 
\label{deg6uni.prop}
Let $(A,\sigma)$ be an exponent $2$ and degree $6$ algebra with unitary involution over $F$. Let $L$ be a finite field extension of $F$ of odd degree. 
If $\sigma_L$ is isotropic, then $\sigma$ is isotropic. 
\end{theorem} 
\begin{proof}
Let $F'$ be the center of the algebra $A$ and let $L'=LF'$. Since the case of a split algebra was discussed in \ref{hypiso.sec} we may assume that $A$ is non-split and therefore that $A=M_3(D)$ for some quaternion division algebra $D$ over $F'$. The involution $\sigma_L$ is an $L'/L$ unitary involution on $A_L$ which is isotropic by assumption. Since $D$ remains division over $L'$, the anisotropic part of $(A,\sigma)_L$  has degree $2$ and therefore, it is isomorphic to $(D_{L'},\tau)$ for some $L'/L$-unitary involution $\tau$ on $D_{L'}$. By a result of Albert (see~\cite[(2.22)]{KMRT}), there exists a unique quaternion algebra $Q$ over $L$ such that $(D_{L'},\tau)=(Q,\bar{\ })\otimes (L',\bar{\ })$, where $\bar{\ }$ denotes the respective canonical involutions on the quaternion algebra $Q$ and the quadratic extension $L'$ over $L$. 
Now, as explained in~\cite[p. 129]{KMRT}, the quaternion algebra $Q$ is the discriminant algebra of $(D_{L'},\tau)$, hence it is Brauer-equivalent to the discriminant algebra of $(A,\sigma)_L$. Since $L/F$ has odd degree, it follows that the discriminant algebra of $(A,\sigma)$ is Brauer equivalent to a quaternion algebra $Q_0$ over $F$ such that $(Q_0)_L\simeq Q$. By~\cite[Prop. 10.30]{KMRT}, $(Q_0)_{F'}$ is Brauer-equivalent to $A$, so $A=M_3(Q_0)_{F'}$ admits a unique involution $\sigma_0$ Witt-equivalent to $(Q_0,\bar{\ })\otimes_F (F',\bar{\ })$. The involutions $\sigma$ and $\sigma_0$ are isomorphic over $L$, hence they are isomorphic over $F$ by~\cite[Theorem 3.2]{Barquero}, hence $\sigma$ is isotropic. 
\end{proof} 

\providecommand{\href}[2]{#2}

\bibliographystyle{amsplain}
\end{document}